\documentclass[11pt]{article}
\usepackage[T1]{fontenc}
\usepackage{amsfonts}
\usepackage{amsmath}
\usepackage{amssymb}
\usepackage{amsthm}
\usepackage{bbm}
\usepackage{bm}
\usepackage{mathrsfs}
\usepackage{verbatim}
\usepackage{setspace}
\usepackage{color}
\usepackage{pdfsync}
\usepackage{enumitem}
\usepackage{graphicx}

\theoremstyle{plain}
\newtheorem{theorem}{Theorem}[section]
\newtheorem{proposition}[theorem]{Proposition}
\newtheorem{lemma}[theorem]{Lemma}
\newtheorem{corollary}[theorem]{Corollary}

\theoremstyle{definition}
\newtheorem{definition}[theorem]{Definition}
\newtheorem{remark}[theorem]{Remark}

\newtheorem{example}[theorem]{Example}

\theoremstyle{remark}

\renewenvironment{thebibliography}[1]{%
\begin{oldthebibliography}{#1}%
\setlength{\baselineskip}{.9em}
\linespread{1}
\small
\setlength{\parskip}{0.3ex}%
\setlength{\itemsep}{.5em}%
}%
{%
\end{oldthebibliography}%
}
\newcommand{\eps}{\varepsilon}

\newcommand{\G}{\mathbb{G}}

\newcommand{\N}{\mathbb{N}}

\newcommand{\R}{\mathbb{R}}

\newcommand{\cC}{\mathcal{C}}

\newcommand{\cE}{\mathcal{E}}
\newcommand{\cF}{\mathcal{F}}
\newcommand{\cG}{\mathcal{G}}

\newcommand{\cI}{\mathcal{I}}

\newcommand{\cT}{\mathcal{T}}

\newcommand{\as}{\mbox{-a.s.}}

\newcommand{\1}{\mathbf{1}}

\numberwithin{equation}{section}

\usepackage[pdfborder={0 0 0}]{hyperref}
\hypersetup{
  urlcolor = black,
  pdfauthor = {Marcel Nutz},
  pdfkeywords = {Mean Field Game; Optimal Stopping; Bank-Run},
  pdftitle = {A Mean Field Game of Optimal Stopping},
  pdfsubject = {A Mean Field Game of Optimal Stopping},
  pdfpagemode = UseNone
}

\begin{document}

\title{\vspace{-0em}
  A Mean Field Game of Optimal Stopping
  \date{First version: May 30, 2016. This version: \today}
\author{
Marcel Nutz%
  \thanks{
  Departments of Statistics and Mathematics, Columbia University, mnutz@columbia.edu. Research supported by an Alfred P.\ Sloan Fellowship and NSF Grant DMS-1512900. This work has enormously profited from discussions with Bruno Bouchard, Ren\'e Carmona, Ioannis Karatzas, Daniel Lacker, Jos\'e Scheinkman, Nizar Touzi, and detailed comments from two anonymous referees, to whom the author is most grateful.
  }
 }
}

\maketitle \vspace{-1.2em}

\begin{abstract}
  We formulate a stochastic game of mean field type where the agents solve optimal stopping problems and interact through the proportion of players that have already stopped. Working with a continuum of agents, typical equilibria become functions of the common noise that all agents are exposed to, whereas idiosyncratic randomness can be eliminated by an Exact Law of Large Numbers. Under a structural monotonicity assumption, we can identify equilibria with solutions of a simple equation involving the distribution function of the idiosyncratic noise. Solvable examples allow us to gain insight into the uniqueness of equilibria and the dynamics in the population. 
\end{abstract}

\vspace{0.9em}

{\small


}

\section{Introduction}\label{se:intro}

Stochastic games with a large number $n$ of players are notoriously intractable. Mean field games were introduced by Lasry and Lions~\cite{LasryLions.06a,LasryLions.06b,LasryLions.07} and Huang, Malham\'e, and Caines~\cite{HuangMalhameCaines.07, HuangMalhameCaines.06} to study Nash equilibria in the limiting regime where $n$ tends to infinity and the players interact symmetrically through the empirical distribution of the private states of all players. Given such a distribution $\mu$, each player typically solves a standard control problem; that is, controls a diffusion while paying some cost of effort. On the other hand, the reward (and possibly the diffusion) depend on $\mu$, which is in turn determined by the actions of all agents. In the analytic theory, such a system is described by a coupled system of nonlinear partial differential equations (PDEs): a Hamilton--Jacobi--Bellman equation describes the optimal control problem when $\mu$ is given, and a Kolmogorov-type equation describes the evolution of~$\mu$ over time as a result of the optimal controls. One of the major difficulties is that the former equation naturally starts from a terminal condition and runs backward in time, whereas the latter runs forward to ensure the consistency of $\mu$; we refer to \cite{Cardaliaguet.13,GamesSaude.14} for background. In a probabilistic version of the theory, the stochastic maximum principle is used and the system of PDEs is replaced by a coupled forward-backward stochastic differential equation; cf.\ \cite{BensoussanFrehseYam.13, CarmonaDelaRue.13, CarmonaDelaRue.14, CarmonaDelaRue.15}. In the simplest case, the agents are exposed to idiosyncratic i.i.d.\ noise (essentially, an independent Brownian motion for each diffusion equation) and thus the equilibria are formulated as deterministic. More recently, the presence of an additional common noise and stochastic equilibria have received considerable attention; see \cite{CarmonaDelaRueLacker.13,CarmonaLacker.15,Fischer.14,Lacker.14,Pham.16}.
A wide range of applications from production models to population dynamics  have emerged over the last decade, several of them summarized in \cite{GueantLasryLions.11}; see also \cite{CarmonaFouqueSun.13} for a recent model of systemic risk and \cite{CarmonaLacker.15} for price impact in finance.  

While mean field games were introduced as a tractable model for a large stochastic game, they are still rather complex. To the best of our knowledge, the only case that can be solved explicitly is linear-quadratic control (linear dynamics, quadratic cost). This situation has been studied in detail; see \cite{Bardi.12, BardiPriuli.14, BensoussanSungYamYung.16, CarmonaFouqueSun.13}. In other cases, one generally has to settle for an abstract description by a coupled system of nonlinear equations.

The main aim of the present paper is to formulate a tractable game of mean field type where the properties of equilibria can be understood somewhat more directly. In our case, the agents will be solving optimal stopping problems rather than diffusion control\footnote{
The possible interest of such a game was first pointed out to the author by Ren\'e Carmona. Section~2 of \cite{GueantLasryLions.11} can be seen as a predecessor, at least in spirit: in a toy example called ``When Does the Meeting Start?'' the agents indirectly control their arrival times at a prespecified location. 

}.
While in a standard mean field game the (spatial) location of the players matters, the state space here becomes binary: each player either has stopped or is still in the game, and the interaction occurs through the number of players that have already stopped. This structure seems appealing due to its simple interpretation and a wide range of possible applications from bank-run models to traffic optimization. On the other hand, it produces an inherent discontinuity in the game: as is well-known in economics (e.g., \cite{DiamondDybvig.83, MorrisShin.04}), games of optimal timing may easily degenerate in that all players stop at the same time. Thus, one of the challenges is to produce a class of models where typical equilibria are non-trivial.

Specifically, we shall study a continuous-time stochastic game with a continuum of players. In equilibrium, each agent $i$ will be solving an optimal stopping problem of the form
$$%
    \sup_{\tau} E\bigg[\exp\bigg(\int_{0}^{\tau} r_{s}\,ds\bigg) \1_{\{\theta>\tau\}\cup \{\theta=\infty\}}\bigg];
$$%
it has two competing forces. The process $r$ can be interpreted as a reward or \emph{interest rate} that is accrued as long as the agent does not stop, thus incentivizing the agent to stay in the game. On the other hand, there is a random time $\theta$ of \emph{default} (of the interest-paying institution): the agent will lose everything if $\theta$ happens before she leaves the game. While the default happens as a ``surprise'' to the agent, the distribution of $\theta$ is governed by an intensity process $\gamma^{i}$ that is known to the agent: the larger $\gamma^{i}$, the more likely it is that default happens soon. More precisely, $\theta$ is modeled as the first jump time of a Cox process with intensity $\gamma^{i}$. This leads to a tractable solution of the single-agent optimal stopping problem---we are taking inspiration from the finance literature 
(e.g., \cite[Chapter~5]{Lando.09}) where it is well-known that a defaultable bond in a similar setting will be priced  just like a non-defaultable one, but with an adjusted interest rate $r-\gamma^{i}$.

The agents are heterogeneous in their views on the distribution of the default---we think of the intensity $\gamma^{i}$ as depending on the subjective probability used by agent $i$. As a result, the players face different optimal stopping problems and may stop at different times. The agents' views on the default intensity will also be influenced by how many players have already stopped; more precisely, the proportion $\rho_{t}\in[0,1]$ of players that have left the game by time $t$. This process is observed by all agents and creates an interaction of mean field type: if $\rho_{t}$ is larger, the intensity of any player will also be larger, meaning that the perceived default will happen sooner. As in bank-run models, this represents that the default of the institution is more likely if more customers have abandoned ship.

While we defer the general formulation of the setting to Section~\ref{se:generalModel}, a typical model may postulate that $\gamma^{i}$ is of the form
\begin{equation}\label{eq:introAdditive}
  \gamma^{i}_{t} = X_{t} + Y^{i}_{t} + c\rho_{t}.
\end{equation}
Here $X$ plays the role of a common noise (the same for all agents) whereas $Y^{i}$ is an idiosyncratic noise that will be i.i.d.\ within the population. Depending on the application, one may interpret $X$ and $Y^{i}$ as public and private signals, respectively, or see their sum as a noisy observation of the true  signal $X$. Moreover, the constant $c\geq0$ governs the strength of interaction; that is, how much the agents' views are affected by $\rho_{t}$. 
Suppose that $\tau^{i}$ is the stopping time chosen by agent $i$, and that the continuum of agents is represented by an atomless probability space $(I,\cI,\lambda)$. Then,
\begin{equation}\label{eq:introEquilib}
  \rho_{t}(\omega) = \lambda\{i:\, \tau^{i}(\omega) \leq t\}
\end{equation}
is the ``proportion'' of players that have stopped prior to time $t$. This can also be seen as the cumulative distribution function (c.d.f.) at time $t$ of the empirical measure that describes the evolution of the system on $I\times \{0,1\}$, recording for each agent $i$ whether stopping has occurred (1) or not (0).

If we start with a given process $\rho$, the intensities $\gamma^{i}$ of the agents are determined. Let us suppose that the associated optimal stopping problems have solutions $(\tau^{i})_{i\in I}$. Tacitly assuming a suitable measurability, we may then consider the process $\lambda\{i:\, \tau^{i}(\omega) \leq t\}$, and if it satisfies~\eqref{eq:introEquilib}, we shall say that $\rho$ and $(\tau^{i})_{i\in I}$ form an equilibrium. Since we are working with a continuum of players, the decision of a single agent does not influence $\rho$, and hence this notion corresponds to a Nash equilibrium: given the strategies of the other players, each player is behaving optimally.

Our main result (Theorem~\ref{th:general}) relates equilibria $\rho_{t}$ to the solution of a finite-dimensional equation. For instance, in the case of~\eqref{eq:introAdditive}, it reads 
\begin{equation}\label{eq:masterIntro}
  1-u = F_{t}(r-x-cu),\quad u\in [0,1],
\end{equation}
where $F_{t}$ is the c.d.f.\ of the idiosyncratic noise $Y_{t}$ and $r$ is the (constant) interest rate. If $\rho(t,x)$ is the (maximal) solution $u$ at time $t$ and $X_{t}$ is the common noise, then $\rho(t,X_{t})$ describes an equilibrium, and a converse is also established. This simple equation allows us to understand the structure and multiplicity of equilibria in some detail. Two ingredients are important for the tractability of our setting. On the one hand, we use an Exact Law of Large Numbers to completely eliminate idiosyncratic randomness (the associated mathematical setup will be discussed later on). This idea and its many incarnations are well-known in economics; see \cite{Aumann.64, DuffieSun.10, KaratzasShubikSudderth.94, Sun.06} to cite but a few examples. On the other hand, we impose the structural assumption that $\gamma^{i}-r$ is increasing, and as we shall see, that leads to a simple solution of the single-agent problem. In comparison to the coupled forward-backward system of equations that is common in the literature on mean field games, one may say that the Hamilton--Jacobi--Bellman part becomes irrelevant because we know the solution of the single-agent problem (given $\rho$) in feedback form, whereas our equation represents the Kolmogorov forward equation---indeed, we may take time derivatives in~\eqref{eq:masterIntro} to find a PDE for $\rho$, or an ODE in the case without common noise.

While the main aim of the present paper is to formulate a tractable example of a mean field game of optimal stopping, there are important aspects that are not discussed. Three major questions are the passage to the limit from a game with $n$ players, what happens when the monotonicity condition is dropped, and the analysis of applications. Recent results on these can be found in \cite{CarmonaDelaRueLacker.17}, and we would like to emphasize that much of that work was carried out in parallel or before ours. In particular, it is shown that in our model of Section~\ref{se:additiveModel}, equilibrium strategies from the continuum formulation are $\eps$-equilibrium strategies in an $n$-player game with large~$n$. Moreover, a general framework for mean field games of optimal stopping (or ``timing'') is introduced and analyzed, and applications to bank run models are discussed.

The remainder of this paper is organized as follows. In the next section, we describe the game more rigorously and analyze the single-player problem in detail, whereas Section~\ref{se:exactLLN} introduces the mathematical setting that allows for an Exact Law of Large Numbers. Section~\ref{se:toyModel} analyzes an insightful toy model without common noise; we discuss examples of (non-)uniqueness and the impact of noise and strength of interaction on the continuity of equilibria in time. Finally, Section~\ref{se:generalModel} treats the general model with common noise.

\section{Description of the Game}\label{se:basicSetup}

Let $(I,\cI,\lambda)$ be a probability space; each $i\in I$ will correspond to an agent. Moreover, let $(\Omega,\cF,P)$ be another probability space, to be used as the sample space. We suppose that $(\Omega,\cF,P)$ is equipped with right-continuous filtrations $\G^{i}=(\cG^{i}_{t})_{t\in\R_{+}}$ and an exponentially distributed random variable $\cE$ which is independent of $\G^{i}$ for all $i\in I$. We interpret $\G^{i}$ as the information available to agent $i$. Finally, let $r$ be a real-valued and locally integrable process (i.e., Lebesgue-integrable on bounded intervals) which is $\G^{i}$-progressively measurable for all $i\in I$; that is, observed by all agents.

\subsection{Single-Agent Problem}

We first consider the optimization problem for a fixed agent $i\in I$ and we denote by $\cT^{i}$ the set of all $\G^{i}$-stopping times.
Let $\gamma^{i}\geq0$ be a $\G^{i}$-progressively measurable process which is locally integrable  %
and consider the random time
$$
  \theta^{i} = \inf\bigg\{t:\, \int_{0}^{t} \gamma^{i}_{s}\,ds = \cE\bigg\}.
$$
One may think of $\theta^{i}$ as the first jump time of a Cox process with intensity~$\gamma^{i}$. The default time $\theta^{i}$ depends on $i$, which will allow us to write all optimal stopping problems under a common probability measure $P$. Alternately, we could deal with a single random time on a canonical space and endow the agents with subjective probabilities $P^{i}$. We have found the former solution easier to write, and they are equivalent in that the agents' decisions (and the equilibria) only depend on the distribution of the intensity.
For technical reasons, we shall assume that  
\begin{equation}\label{eq:intCond}
    \left. 
    \begin{matrix}
    \text{$r^{+}$ is integrable on $[0,\infty)$, $P$-a.s.,} \;\\[.1em]
    \text{or}\\[.1em]
    \text{ $\inf\{t: \, \gamma^{i}_{t}-r_{t}\geq 0\}<\infty$, $P$-a.s.} \;
    \end{matrix}
    \right\}
\end{equation} 
See also Remark~\ref{rk:noOptimizer} below for the necessity of such a condition. We then have the following result on the single-agent problem.

\begin{lemma}\label{le:singleAgent}
  Suppose that $\gamma^{i}-r$ is increasing\,\footnote{Increase is to be understood in the non-strict sense throughout the paper.} and~\eqref{eq:intCond} holds. Then,
  $$
  \tau^{i} := \inf\{t: \, \gamma^{i}_{t}-r_{t}\geq 0\} \in\cT^{i}
  $$
  is a solution of the optimal stopping problem\,\footnote{
  We use the convention that $\int_{0}^{\infty} r_{s}\,ds:=-\infty$ if $\int_{0}^{\infty} r^{+}_{s}\,ds=\int_{0}^{\infty} r^{-}_{s}\,ds=\infty$.
  }
  \begin{equation}\label{eq:optStopProblem}
    \sup_{\tau\in\cT^{i}} E\bigg[\exp\bigg(\int_{0}^{\tau} r_{s}\,ds\bigg) \1_{\{\theta^{i}>\tau\}\cup \{\theta^{i}=\infty\}}\bigg] . 
  \end{equation}
  If the value of \eqref{eq:optStopProblem} is finite, then $\tau^{i}$ is minimal among all solutions, and if, in addition, $\gamma^{i}-r$ is strictly increasing, then $\tau^{i}$ is the unique solution.
\end{lemma}

\begin{proof}
  Due to the increase of $\gamma^{i}-r$ and the right-continuity of $\G^{i}$, the right limit process $\zeta$ of $\gamma^{i}-r$ exists and is $\G^{i}$-progressively measurable. As $\{\tau^{i}\leq t\} = \{\zeta_{t}\geq0\}\in\cG^{i}_{t}$, we have $\tau^{i}\in\cT^{i}$. 
  
  Let $\tau\in\cT^{i}$ be such that $r^{+}$ is integrable on $[0,\tau)$, $P$-a.s. Using also the independence of $\cE$ and $\G^{i}$, we have
  \begin{align*}
  P\bigg[\{\theta^{i}>\tau\}\cup \{\theta^{i}=\infty\}\bigg|\cG^{i}_{\tau}\bigg]    
  &= P\bigg[\int_{0}^{\tau} \gamma^{i}_{s}\,ds < \cE\bigg|\cG^{i}_{\tau}\bigg] \\
  &= E\bigg[\exp \bigg(-\int_{0}^{\tau} \gamma^{i}_{s}\,ds\bigg)\bigg|\cG^{i}_{\tau}\bigg].
\end{align*}
Hence,
\begin{align*}
  E\bigg[\exp\bigg(\int_{0}^{\tau} r_{s}\,ds\bigg) \1_{\{\theta^{i}>\tau\}\cup \{\theta^{i}=\infty\}}\bigg|\cG^{i}_{\tau}\bigg]    
= E\bigg[\exp\bigg( \int_{0}^{\tau} (r_{s}-\gamma^{i}_{s})\,ds \bigg)\bigg|\cG^{i}_{\tau}\bigg]
\end{align*}
and finally
\begin{equation}\label{eq:adjustedInterest}
  E\bigg[\exp\bigg(\int_{0}^{\tau} r_{s}\,ds\bigg) \1_{\{\theta^{i}>\tau\}\cup \{\theta^{i}=\infty\}}\bigg] = E\bigg[\exp\bigg( \int_{0}^{\tau} (r_{s}-\gamma^{i}_{s})\,ds \bigg)\bigg].
\end{equation}
If we are in the first case of~\eqref{eq:intCond}, our integrability condition holds for all $\tau\in\cT^{i}$ and as $r-\gamma^{i}$ is decreasing, the representation on the right-hand side shows that $\tau^{i}$ is optimal. In the second case of~\eqref{eq:intCond}, as $r$ is locally integrable, we still have~\eqref{eq:adjustedInterest} for every finite-valued $\tau\in\cT^{i}$, and we deduce that $\tau^{i}$ is optimal among all those stopping times. If $\tau$ is a general stopping time and $N\in\N$, Fatou's lemma and that optimality yield
\begin{align*}
  E\bigg[\exp\bigg(\int_{0}^{\tau}\! r_{s}\,ds\bigg) \1_{\{\theta^{i}>\tau\}\cup \{\theta^{i}=\infty\}}\bigg] 
  & \leq \liminf_{N\to\infty} E\bigg[\exp\bigg(\int_{0}^{\tau\wedge N} \!\! r_{s}\,ds\bigg) \1_{\theta^{i}>\tau\wedge N}\bigg] \\
  & \leq  E\bigg[\exp\bigg(\int_{0}^{\tau^{i}} \! r_{s}\,ds\bigg) \1_{\theta^{i}>\tau^{i}}\bigg],
\end{align*}
so that $\tau^{i}$ is in fact optimal among all stopping times.
The remaining assertions can also be inferred from~\eqref{eq:adjustedInterest}. 
\end{proof}

We see from the proof of Lemma~\ref{le:singleAgent} how the increase of $\gamma^{i}-r$ leads to a simple solution of the optimal stopping problem and that will contribute greatly to the tractability of equilibria. We have little else to say in defense of that condition.

\begin{remark}\label{rk:noOptimizer}
  As usual in infinite-horizon stopping problems, an integrability assumption is necessary to ensure existence of an optimal stopping time. In particular, if $r>\gamma^{i}>0$ are constant, then $\tau^{i}=\infty$ which is clearly not optimal as then $P(\{\theta^{i}>\tau^{i}\}\cup\{\theta^{i}=\infty\})=0$.
    If we consider the same problem with a  horizon $T\in(0,\infty)$, no extra assumption is necessary. %
\end{remark}

\begin{remark}\label{rk:strictIncrease}
  In Lemma~\ref{le:singleAgent} and the remainder of this paper, we use strict monotonicity of $\gamma^{i}-r$ as a simple sufficient condition for the uniqueness of~$\tau^{i}$. In specific cases one may want to use a sharper condition; for instance, discrete-time problems can be embedded in our results by using piecewise constant processes, but then the notion of strict monotonicity needs to be adapted.
\end{remark}

\subsection{Interaction}

While the $i$th agent chooses to stop at $\tau^{i}\in\cT^{i}$, the agents will interact through the ``proportion'' of agents that have already stopped. Indeed, we shall specify $\gamma^{i}$ as a functional depending on a process $\rho$, and then an \emph{equilibrium} will be a collection of stopping times $\tau^{i}\in\cT^{i}$ which solve~\eqref{eq:optStopProblem} for $\lambda$-almost all $i\in I$ and such that 
\begin{equation}\label{eq:equilibDef}
  \rho_{t} = \lambda\{i:\, \tau^{i}\leq t\}.
\end{equation}
When $\lambda$ is atomless, the decision of a single agent does not influence this quantity and hence we indeed have a Nash equilibrium. Clearly, the process~$\rho$ will necessarily be increasing and $[0,1]$-valued. Moreover, we think of $\rho$ as being observed by all agents, so $\rho$ will be $\G^{i}$-adapted for all $i$.

In~\eqref{eq:equilibDef}, we are tacitly assuming that the set on the right-hand side is $\cI$-measurable $P$-a.s., which is highly nontrivial for a continuum of i.i.d.\ random variables. The setup that can guarantee this is discussed in the next section. Before that, however, let us illustrate the concepts introduced thus far by a very simple example where the agents do not use any signals except $\rho$.

\begin{example}[Sunspot]\label{ex:sunspot} 
  (i) Let $\lambda$ be the Lebesgue measure on $I=[0,1]$, let $r>0$ be constant and let $X$ be a right-continuous, increasing process on $(\Omega,\cF,P)$, progressively measurable for the common, right-continuous filtration $\G^{i}=\G$ (i.e., the same for all agents) and such that $X_{\infty}>1$. Suppose agent $i\in [0,1]$ believes in the intensity
$$
\gamma^{i}_{t} = (r - i + \rho_{t})\vee 0.
$$
Thus, $i$ acts as an index of ``optimism'' or ``risk tolerance''---agents with higher index believe that $\theta^{i}$ will happen later.
We claim that 
$$
  \rho_{t}=(X_{t}\wedge 1)\vee 0
$$
yields an equilibrium. Indeed, the optimal stopping times are then given by
$$
  \tau^{i} = \inf\{t: \, \gamma^{i}_{t}-r \geq 0\}  = \inf \{t:\,  \rho_{t}=i\}=\inf \{t:\,  X_{t}\geq i\}
$$
and that results in
$$
  \lambda \{i:\, \tau^{i}\leq t\} = \lambda \{i:\,  X_{t}\geq i\} =(X_{t}\wedge 1)\vee 0 = \rho_{t};
$$
note that the second condition of~\eqref{eq:intCond} is satisfied.

For instance, the choice $X_{t}=t$ gives rise to $\rho_{t}=t\wedge1$ and $\tau^{i}=i$, showing that the agents stop at deterministic times which are uniformly distributed over the time interval $[0,1]$. If $X_{0}$ is strictly positive, we see that some of the agents stop instantaneously at $t=0$, whereas if $X_{0}$ is strictly negative, it will take a while before any agents stop.

(ii) A similar equilibrium exists in a finite player game. Let $n\in\N$ and let $\lambda$ be the normalized counting measure on $I=\{1/n,2/n,\ldots,1\}$; this corresponds to $n$ equally weighted agents.  In the same setting as in (i), an equilibrium is described by $\tau^{i}=\inf \{t:\,  X_{t}\geq i\}$ and
$$
  \rho_{t}=\lfloor (X_{t}\wedge 1)\vee 0 \rfloor,
$$
where $\lfloor x \rfloor := \max \{s\in I:\, s\leq x\}$.
\end{example}

\begin{remark}\label{rk:ex}
  (i) In the preceding example, the process $X$ is not part of the functional form of $\gamma^{i}$; essentially, \emph{any} process $X$ gives rise to an equilibrium. The interpretation is that if all agents agree that some commonly observed signal $X$ is relevant, it indeed becomes relevant---the name ``sunspot equilibrium'' suggests itself. We shall see in Example~\ref{ex:uniformInitial} that this situation is a degenerate limit of a model where uniqueness is the typical case.
  
  (ii) If all agents are perfectly identical, the problem will degenerate since they will (typically) all stop at the same time. Thus, in the above example, the agents have been made heterogeneous by varying the risk tolerance. This is not necessary when the agents are already heterogeneous due to private signals, as in the later sections.
\end{remark}

\section{Mathematical Setting and Exact Law of Large Numbers}\label{se:exactLLN}

In this section, we introduce the setting to accommodate a continuum of agents and their private signals.
Let $(I,\cI,\lambda)$ be an atomless (hence, uncountable) probability space and let $(\Omega,\cF,P)$ be another probability space.

\begin{definition}\label{de:essentialPairwiseIndep}
  A family $(f_{i})_{i\in I}$ of random variables on $(\Omega,\cF,P)$ is \emph{essentially pairwise  independent} if for $\lambda$-almost all $i\in I$, $f_{i}$ is independent of $f_{j}$ for $\lambda$-almost all $j\in I$. The family is \emph{essentially pairwise i.i.d.}\ if, in addition, all $f_{i}$ have the same distribution. Analogously,  for a $\sigma$-field $\cC\subseteq\cF$, the family $(f_{i})_{i\in I}$ is \emph{essentially pairwise conditionally independent} given $\cC$, if for $\lambda$-almost all $i\in I$, $f_{i}$ is conditionally independent of $f_{j}$ given $\cC$ for $\lambda$-almost all $j\in I$.
\end{definition}

In what follows, we need to work on a probability space that is larger than the usual product\footnote{
Here and below, we use the convention that the product $\sigma$-field $\cI\otimes \cF$ is completed. 
}
 $(I\times\Omega,\cI\otimes \cF,\lambda\otimes P)$, because the latter does not support relevant families of i.i.d.\ random variables. More precisely, we have the following fact; see, e.g., \cite[Proposition~2.1]{Sun.06}.

\begin{remark}\label{rk:productDegenerate}
  If $f: I\times\Omega\to\R$ is an $\cI\otimes \cF$-measurable function such that $f(i,\cdot)$, $i\in I$ are essentially pairwise i.i.d., then $f$ is constant $\lambda\otimes P$-a.s.
\end{remark}

Following \cite{Sun.06}, we say that a probability space $(I\times\Omega,\Sigma,\mu)$ is an \emph{extension} of the  product $(I\times\Omega,\cI\otimes \cF,\lambda\otimes P)$ if $\Sigma$ contains $\cI\otimes \cF$ and the restriction of $\mu$ to $\cI\otimes \cF$ coincides with $\lambda\otimes P$. It is a \emph{Fubini extension} if, in addition, any $\mu$-integrable\footnote{
That is, $f$ is measurable for the $\mu$-completion of $\Sigma$ and $\int |f|\,d\mu<\infty$.
}
function $f: I\times\Omega\to\R$ satisfies the assertion of Fubini's theorem\footnote{
Since $\Sigma$ may be strictly larger than $\cI\otimes \cF$, this is not automatic.
}; that is,

\begin{enumerate}
 \item for $\lambda$-almost all $i\in I$, the function $f(i,\cdot)$ is $P$-integrable,
 \item for $P$-almost all $\omega\in \Omega$, the function $f(\cdot,\omega)$ $\lambda$-integrable,
 \item $i\mapsto \int f(i,\cdot)\,dP$ is $\lambda$-integrable, $\omega\mapsto \int (\cdot,\omega)\,d\lambda$ is $P$-integrable, and
 $$
   \int f\,d \mu = \iint f(i,\omega)\,P(d\omega)\,\lambda(di) = \iint f(i,\omega)\,\lambda(di)\,P(d\omega).
 $$
\end{enumerate}

Let $(I\times\Omega,\Sigma,\mu)$ be a Fubini extension of $(I\times\Omega,\cI\otimes \cF,\lambda\otimes P)$. Then, essentially pairwise independent families satisfy an exact version of the Law of Large Numbers. The simplest version runs as follows---more generally, an exact version of the Glivenko--Cantelli Theorem holds; cf.\ \cite[Corollary~2.9]{Sun.06}.

\begin{proposition}[Exact Law of Large Numbers]\label{pr:LLN}
  Let $f: I\times\Omega\to \R$ be $\mu$-integrable. If $f(i,\cdot)$, $i\in I$ are essentially pairwise i.i.d.\ with a distribution having mean $m$, then $\int f(\cdot,\omega)\,d\lambda=m$ for $P$-almost all $\omega\in\Omega$.
\end{proposition}  

We shall also need a conditional version as provided by \cite[Corollary~2]{QiaoSunZhang.14}.

\begin{proposition}[Conditional Exact Law of Large Numbers]\label{pr:condLLN}
  Let $\cC\subseteq\cF$ be a countably generated $\sigma$-field and let $f: I\times\Omega\to \R$ be $\mu$-integrable. If $f(i,\cdot)$, $i\in I$ are essentially pairwise conditionally independent given $\cC$, then $\int f(\cdot,\omega)\,d\lambda=\int E^{\mu}[f|\cI\otimes\cC](\cdot,\omega)\,d\lambda$ for $P$-almost all $\omega\in\Omega$. 
\end{proposition}

In view of Remark~\ref{rk:productDegenerate}, it is not obvious that the preceding propositions are not vacuous---that is guaranteed by the next two results.

The space $(I\times\Omega,\Sigma,\mu)$ is called \emph{rich} if there exists a $\Sigma$-measurable function $f: I\times\Omega\to\R$ such that $f(i,\cdot)$, $i\in I$ are essentially pairwise i.i.d.\ with a uniform distribution on $[0,1]$.
Like an atomless probability space supports random variables with any given distribution, a rich Fubini extension supports essentially pairwise i.i.d.\ families with any given distribution; cf.\ \cite[Corollary~5.4]{Sun.06}.

\begin{lemma}\label{le:richSupportsiid}
  Let $(I\times\Omega,\Sigma,\mu)$ be a rich Fubini extension of $(I\times\Omega,\cI\otimes \cF,\lambda\otimes P)$, let $S$ be a Polish space and let $\nu$ be a Borel probability measure on $S$. There exists a $\Sigma$-measurable function $f:I\times\Omega\to S$ such that $f(i,\cdot)$, $i\in I$ are essentially pairwise independent and $f(i,\cdot)$ has distribution $\nu$ for all $i\in I$.
\end{lemma}

\begin{lemma}\label{le:richExtensionsExist}
  There exist atomless probability spaces $(I,\cI,\lambda)$ and $(\Omega,\cF,P)$ such that $(I\times\Omega,\cI\otimes \cF,\lambda\otimes P)$ admits a rich Fubini extension.
\end{lemma}

This is part of the assertion of \cite[Proposition~5.6]{Sun.06} which also shows that one can take $I=[0,1]$ and $\Omega=\R^{[0,1]}$. The main result of \cite{SunZhang.09} shows that, in addition, one can take $\lambda$ to be an extension of the Lebesgue measure (but not the Lebesgue measure itself). A different construction, avoiding nonstandard analysis, is presented in \cite{Podczeck.10}.

\section{A Toy Model}\label{se:toyModel}

In this section, we discuss a simple setting where the agents' signals are i.i.d.; that is, pure idiosyncratic noise. While not suitable for most applications, this will allow us to explain the effect of the Exact Law of Large Numbers in our model and to discuss some finer questions of uniqueness and nondegeneracy without too many distractions.

Consider the setup introduced in Section~\ref{se:basicSetup} with atomless probability spaces $(I,\cI,\lambda)$ and $(\Omega,\cF,P)$, and let $(I\times\Omega,\Sigma,\mu)$ be a Fubini extension of their product. For each $i\in I$, let $Y^{i}\geq0$ be a right-continuous, increasing, $\G^{i}$-progressively measurable process. We assume that for each $t\geq0$,  $(i,\omega)\mapsto Y^{i}_{t}(\omega)$ is $\Sigma$-measurable and that $Y^{i}_{t}$, $i\in I$ are essentially pairwise i.i.d. Moreover, we assume that the distribution of $Y^{i}_{t}$ has no atoms; that is, its c.d.f.\ $y\mapsto F_{t}(y):=P\{Y^{i}_{t}\leq y\}$ is continuous.

\begin{proposition}\label{pr:privateInfo}
  Let $r\in\R$ and $c\in\R_{+}$. The equation
  \begin{equation}\label{eq:privateInfo}
    1-u = F_{t}(r-cu),\quad u\in [0,1]
  \end{equation}
  has a maximal solution $\rho(t)\in[0,1]$ for every $t\geq0$, and $t\mapsto \rho(t)$ is right-continuous. Define also 
  $$
    \gamma^{i}_{t} = Y^{i}_{t} + c\rho(t), \quad \tau^{i} = \inf \{t:\, Y^{i}_{t} + c\rho(t) =r\},
  $$  
  and assume that \eqref{eq:intCond} is satisfied for all $i$.
  
  (i) Then, $\rho$ and $(\tau^{i})_{i\in I}$ define an equilibrium: $\tau^{i}\in\cT^{i}$ is an optimal stopping time for agent $i$, the mapping $(i,\omega)\mapsto \tau^{i}(\omega)$ is $\Sigma$-measurable, and 
  $$
    \lambda\{i:\,\tau^{i}\leq t\}=\rho(t)\quad P\as \quad \mbox{for all} \quad t\geq0.
  $$  
  
  (ii) Conversely, let $\bar\rho$ be a right-continuous function corresponding to an equilibrium. If $\gamma^{i}$ is strictly increasing for all $i$, then $\bar\rho(t)$ is a solution of~\eqref{eq:privateInfo} for every $t\geq0$.
\end{proposition}

\begin{proof}[Sketch of Proof.]
  The proposition is a special case of Theorem~\ref{th:general} that will be proved later on, so we shall only explain the most important steps. 
  
  (a) We first argue that $\rho$ is well-defined, increasing and right-continuous.
  Let us consider, for a right-continuous and increasing function $F: \R\to [0,1]$, the zeros of
  $$
   G(u):=F(r-cu)-1+u, \quad u\in[0,1].
  $$
  We have $G(0)=F(r)-1\leq0$ and $G(1)=F(r-c)\geq0$. Moreover, $G$ is left-continuous and its jumps satisfy $\Delta G\leq0$. Thus, $G$ must have at least one zero in $[0,1]$. If $u_{n}\uparrow u$ is a maximizing sequence of zeros in $[0,1]$, then $G(u)=0$ by left-continuity and $u$ is the maximal zero.
  
  Next, write $G_{t}(u):=F_{t}(r-cu)-1+u$ and let $\rho(t)$ be the maximal zero for each $t\geq0$. The increase and the right-continuity of $t\mapsto\rho(t)$ can be inferred from the increase of $Y$ and the right-continuity of $Y$ and the continuity of $y\mapsto F_{t}(y)$, respectively---we defer the details.

  (b) Next, we verify that $\rho$ and $(\tau^{i})_{i\in I}$ determine an equilibrium. It follows from (a) that $\gamma^{i}=Y^{i} + c\rho$ is increasing and right-continuous; hence, Lemma~\ref{le:singleAgent} yields that $\tau^{i}\in\cT^{i}$ is an optimal stopping time for all $i\in I$ and that $\{(i,\omega):\, \tau^{i}(\omega)\leq t\} = \{(i,\omega):\, Y^{i}_{t}(\omega) + c\rho(t)\geq r\}\in \Sigma$. Using the Exact Law of Large Numbers of Proposition~\ref{pr:LLN}, the continuity of $F_{t}$ and the definition of $\rho(t)$, we have $P$-a.s.\ that
  \begin{align*}
    \bar\rho(t):=\lambda\{i:\, \tau^{i}\leq t\} &=\lambda\{i:\, Y^{i}_{t} + c\rho(t)\geq r\} \\
    &= \int \! P\{ Y^{i}_{t} + c\rho(t)\geq r \}\,\lambda(di)= 1-F_{t}(r-c\rho(t))=\rho(t)
  \end{align*}
  for all $t\geq0$.
  
  (c) Let $\bar\rho: \R_{+}\to\R$ be a right-continuous function corresponding to an equilibrium; that is, $\bar\rho(t)=\lambda\{i:\, \tau^{i}\leq t\}$ for some optimal $\tau^{i}\in\cT^{i},$ $i\in I$. Then $\bar\rho$ is clearly increasing and $[0,1]$-valued. Due to the strict increase of~$\gamma^{i}$, we know from Lemma~\ref{le:singleAgent} that $\tau^{i}=\inf\{t: \gamma^{i}_{t}\geq r\}$. Thus,
  \begin{align*}
    \bar\rho(t)=\lambda\{i:\, \tau^{i}\leq t\} &=\lambda\{i:\, Y^{i}_{t} + c\bar\rho(t)\geq r\} \\
    &= \int P\{ Y^{i}_{t} + c\bar\rho(t)\geq r \}\,\lambda(di)= 1-F_{t}(r-c\bar\rho(t));
  \end{align*}
  that is, $\bar\rho(t)$ is a solution of~\eqref{eq:privateInfo} for all $t\geq0$.
\end{proof}

We begin our discussion with some  observations about uniqueness.

\begin{remark}\label{rk:minimalZero}
  (i) Equation~\eqref{eq:privateInfo} may have more than one solution; cf.\ Example~\ref{ex:threeSolutions}. If $t\mapsto\rho(t)$ is any right-continuous, increasing solution of \eqref{eq:privateInfo}, not necessarily maximal, then $\rho$ induces an equilibrium, by the same arguments as in the above proof of Proposition~\ref{pr:privateInfo}.

    (ii) Equation~\eqref{eq:privateInfo} also has a minimal solution, and it is automatically increasing in $t$. However, it is not necessarily right-continuous; see also Example~\ref{ex:uniformInitial}\,(iii) below. Instead, it is left-continuous provided that $Y$ is. If the solution is unique and $Y$ is continuous, the solution is both minimal and maximal, and therefore continuous.
\end{remark}

Next, we analyze a special case of Proposition~\ref{pr:privateInfo} that is explicitly solvable and sheds some light on the impact of the constant $c\geq0$ that parametrizes the strength of interaction. One intuition is that if the interaction between the agents is too strong, some agents' stopping will lead to a domino effect where all others end up stopping immediately after.

\begin{example}\label{ex:uniformInitial}
	Let $r\geq1$ and let $U^{i}$, $i\in I$ be essentially pairwise i.i.d.\ with a uniform distribution on $[r-1,r]$. Moreover, let $a: \R_{+}\to\R_{+}$ be a strictly increasing, right-continuous function with $a(0)=0$ and $a(\infty)>1$---the latter will ensure that~\eqref{eq:intCond} holds. We then consider the strictly increasing process
	$$
	  Y^{i}_{t}=U^{i}+a(t)
	$$
	and note that $F_{t}(y)=F(1+y-a(t)-r)$, where $F$ is the c.d.f.\ of the uniform distribution on $[0,1]$. Thus, Equation~\eqref{eq:privateInfo} becomes
	\begin{equation}\label{eq:uniformInitial}
	  1-u = F(1-cu-a(t)),\quad u\in [0,1].
	\end{equation}
	
	\begin{enumerate}
	\item  \emph{No interaction,} $c=0$. Clearly, the unique solution is $\rho(t)=a(t)\wedge 1$, and this is  the unique equilibrium by the last part of Proposition~\ref{pr:privateInfo}.
	
	\item \emph{Moderate interaction,} $c\in(0,1)$. Then, \eqref{eq:uniformInitial} is easily seen to have a unique solution $\rho(t)\in [0,1]$; namely,
	$$
	  \rho(t)=[(1-c)^{-1}a(t)]\wedge 1,
	$$
	and this is the unique equilibrium. In particular, the population of stopped agents evolves in a nondegenerate, continuous fashion for $c\in(0,1)$. The larger the interaction coefficient $c$, the more agents stop earlier.
  
  \item \emph{Critical interaction,} $c=1$.  Using that $a(t)>0$ for $t>0$, we can check that $\rho(t)= 1$ is the unique solution of~\eqref{eq:uniformInitial} for $t>0$. Thus, $\rho\equiv1$ is the unique right-continuous solution; that is, all agents stop at $t=0$. This is also the unique right-continuous equilibrium.
  
  It is worth noting that any $u\in[0,1]$ is a solution of \eqref{eq:uniformInitial} for $t=0$; recall that $a(0)=0$. Intuitively speaking, solutions $\rho(t)=u\1_{\{0\}}+\1_{(0,\infty)}$ that are not right-continuous correspond to equilibria where a fraction $u$ of the agents stop at time zero whereas the rest stop ``immediately after'' zero. This may illustrate why we have imposed right-continuity in our results.
	
	\item \emph{Supercritical interaction,} $c>1$. We see directly that $\rho(t)= 1$ is the unique solution of~\eqref{eq:uniformInitial} for all $t\geq0$.
  \end{enumerate}
\end{example}

\subsection{On the Multiplicity of Equilibria}

The following is a fairly well-behaved example of non-uniqueness.

\begin{example}\label{ex:threeSolutions} 
  We consider again the setting of Example~\ref{ex:uniformInitial}, with $r=1$ and $a(t)=t$, say, but we now replace the uniform distribution of $U^{i}$ with a measure that assigns mass $\eps\in(0,1/4)$ uniformly to $[0,\eps]$ and to $[1-\eps,1]$, and the remaining mass $1-2\eps$ uniformly to $[1/2-\eps,1/2+\eps]$. For small enough $\eps>0$ and $c\in(0,1)$, we see that Equation~\eqref{eq:privateInfo} has three interior solutions for $t$ within a certain interval, whereas all solutions are at the origin for $t=0$. We can select any of these solutions to form an increasing right-continuous process $\rho$ that corresponds to a legitimate equilibrium.
\end{example}
  
  To understand this bifurcation, let us first look at an even simpler situation where $\gamma^{i}_{t}=r-c+c\rho(t)$ for all $i$. At time $t=0$, two obvious equilibria are: No agent stops, then $\rho(t)=0$ and $\gamma^{i}_{t}=r-c<r$, so it is indeed optimal not to stop. Or, all agents stop immediately, then $\rho(t)=1$ and $\gamma^{i}_{t}=r$, so it is indeed optimal to stop. (This coordination problem is very similar to the phenomenon discussed e.g.\ in \cite{DiamondDybvig.83, MorrisShin.04}.) When $\gamma^{i}_{t}$ is random, a similar choice can  arise at an intermediate time for a subset of the population corresponding to an atom in the distribution of $\gamma^{i}$. More generally, the bifurcation can also happen in a continuous fashion when the random variable is sufficiently concentrated (relative to the size of $c$) around some point rather than having atoms, and this is what was witnessed in Example~\ref{ex:threeSolutions}.

The following observation is a different view on the same interplay.

\begin{remark}\label{rk:implicitFunctionThm}
  Let $(t,y)\mapsto F_{t}(y)$ be $C^{1}$ and write $f_{t}=\partial_{y}F_{t}$ for the probability density at time $t$. Suppose that $\partial_{u} F_{t}(r-cu) \neq -1$; that is,
  $$
   cf_{t}(r-cu)\neq 1
  $$
  for $u$ in a neighborhood of a solution $\rho(t)$ of~\eqref{eq:privateInfo}. Then, the Implicit Function Theorem shows that $\rho$ is locally unique and  $C^{1}$. For $c>0$, this is true, in particular, if $0\leq f_{t} < c^{-1}$ on $[r-c,r]$. Or, put differently: if $F_{t}$ is not too concentrated or if $c$ is small enough, local uniqueness holds.
\end{remark}

The following provides a broader perspective on Example~\ref{ex:sunspot} and shows that uniqueness may fail even more dramatically in certain  regimes.

\begin{example}\label{ex:retrieveSunspot} 
  We consider again the setting of Example~\ref{ex:uniformInitial}, except that we now take 
  $$
  a(t)=\begin{cases}
    0, & t<T\\
    2, & t\geq T,
  \end{cases}
  $$
  where $T\in(0,\infty)$ acts as a time horizon. Indeed, this definition implies $\gamma^{i}_{t}\geq r+1$ for all $t\geq T$ and $i\in I$, so that all agents will stop at $T$, if not earlier.
  
  Thus, we are interested in the situation on $[0,T)$, where \eqref{eq:uniformInitial} becomes
	$$
	  1-u = F(1-cu).
	$$
	
	(i) For $c\in[0,1)$, the unique solution is $u=0$, and thus
	$$
	  \rho(t) = \1_{[T,\infty)}
	$$
	is the corresponding equilibrium: at $t=0$, only a nullset of agents stop, and that does not change until $T$. One can check that this equilibrium is unique, even though $\gamma^{i}$ is not strictly increasing.

	(ii) At the critical value $c=1$, uniqueness is lost and the situation is completely different. Indeed, the equation becomes the tautological $1-u=1-u$. Thus, \emph{any} right-continuous, increasing function $X(t)$ with values in $[0,1]$ determines an equilibrium via
	$$
	  \rho(t) = X(t)\1_{[0,T)} + \1_{[T,\infty)}.
	$$
	
	This is the situation we have encountered in Example~\ref{ex:sunspot}: in terms of the equilibrium distribution, there is equivalence between assigning risk aversion $r-i$ to agent $i$ and sampling uniformly from $[r-1,r]$ for every agent as in the present example. The latter basically corresponds to randomly permuting the labels of the agents in Example~\ref{ex:sunspot}.
\end{example}

\section{The General Model}\label{se:generalModel}

In this section, we generalize the model from the previous section by specifying the intensities $\gamma^{i}$ as a possibly nonlinear function of i.i.d.\ signals $Y^{i}$ and a common signal $X$. As a result, the intensities are conditionally independent rather than independent, and the equilibrium becomes a function of $X$.

As above, we consider the setup introduced in Section~\ref{se:basicSetup} with atomless probability spaces $(I,\cI,\lambda)$ and $(\Omega,\cF,P)$, and let $(I\times\Omega,\Sigma,\mu)$ be a Fubini extension of their product. For each $i\in I$, let $Y^{i}\geq0$ be a right-continuous, increasing, $\G^{i}$-progressively measurable process. We assume that for each $t\geq0$,  $(i,\omega)\mapsto Y^{i}_{t}(\omega)$ is $\Sigma$-measurable and that $Y^{i}_{t}$, $i\in I$ are essentially pairwise i.i.d. Moreover, we assume that the distribution of $Y^{i}_{t}$ has no atoms; that is, its c.d.f.\ $y\mapsto F_{t}(y):=P\{Y^{i}_{t}\leq y\}$ is continuous. In addition, let $X$ be a $d$-dimensional, right-continuous, (componentwise) increasing process which is $\G^{i}$-progressively measurable for all $i$ and such that $X_{t}$ and $Y^{i}_{t}$ are independent for all $t\geq0$. Thus, $X$ is interpreted as public information whereas $Y^{i}$ is a private signal%
\footnote{
Additional idiosyncratic signals could also be included. In particular, a signal at time zero can be used to assign different functional forms of $\gamma^{i}$ to the agents, similarly as at the end of Example~\ref{ex:uniformInitial}.
}
available only to agent $i$. (A slightly different setup and interpretation are discussed in Section~\ref{se:additiveModel}.) 

Let $r: \R_{+}\times\R^{d}\to\R$ be a right-continuous, decreasing function. The interest rate process will be assumed to be of the form\footnote{
Since $X$ can be multivariate, this entails no loss of generality relative to introducing yet another stochastic process.
}
$$
  r_{t}=r(t,X_{t}).
$$
Finally, the  intensity of agent $i$ will be of the form $\gamma^{i}_{t} = g(t,X_{t},Y^{i}_{t},\rho_{t})$, where
$$
  g: \R_{+}\times \R^{d}\times \R\times [0,1]\to\R
$$
is a continuous function with $g(t,X_{t},Y^{i}_{t},0)\geq0$,  increasing in all its arguments and such that for all $(t,x,u)$, $y\mapsto g(t,x,y,u)$ admits an inverse $y'\mapsto g^{-1}(t,x,y',u)$ on its range which we assume to be $\R$ for simplicity. We suppose that $g^{-1}$ is again continuous. 

\begin{theorem}\label{th:general}
  The equation
  \begin{equation}\label{eq:general}
    1-u = F_{t}(g^{-1}(t,x,r,u)),\quad u\in [0,1]
  \end{equation}
  has a maximal solution $\rho(t,x,r)\in[0,1]$ for every $(t,x,r)\in\R_{+}\times\R^{d}\times\R$, and $\rho_{t}:=\rho(t,X_{t},r(t,X_{t}))$ is a right-continuous, increasing process. Define also 
  $$
    \gamma^{i}_{t} = g(t,X_{t},Y^{i}_{t},\rho_{t}), \quad \tau^{i} = \inf \{t:\, \gamma^{i}_{t} =r_{t}\}
  $$  
  and assume that~\eqref{eq:intCond} is satisfied for all $i$.
  
  (i) Then, $\rho$ and $(\tau^{i})_{i\in I}$ define an equilibrium: $\tau^{i}\in\cT^{i}$ is an optimal stopping time for agent $i$, the mapping $(i,\omega)\mapsto \tau^{i}(\omega)$ is $\Sigma$-measurable, and 
  $$
    \lambda\{i:\,\tau^{i}\leq t\}=\rho_{t}\quad P\as \quad \mbox{for all} \quad t\geq0.
  $$  
  More generally, this holds for any measurable solution $\rho(t,x,r)$ of~\eqref{eq:general} such that $\rho(t,X_{t},r(t,X_{t}))$ is right-continuous and increasing.
  
  (ii) Conversely, let $t\mapsto \bar\rho_{t}$ be a right-continuous process corresponding to an equilibrium and suppose that $\bar\rho_{t}=\bar\rho(t,X_{t},r(t,X_{t}))$ for some measurable function $\bar\rho$. If $\gamma^{i}$ is strictly increasing for all $i$, then for every $t\geq0$, $\bar\rho(t,x,r(t,x))$ solves~\eqref{eq:general} for $(P\circ X_{t}^{-1})$-almost all $x\in\R^{d}$.
\end{theorem}

\begin{proof}
  (a) We claim that $\rho(t,x,r)$ is well-defined, increasing in $(t,x)$ and decreasing in $r$, and (jointly) right-continuous in $(t,x)$ and left-continuous in $r$. Indeed, fix $(t,x,r)$ and consider the function
  $$
    G_{t,x,r}(u) :=F_{t}(g^{-1}(t,x,r,u))-1+u, \quad u\in[0,1].
  $$
  Since $F_{t}$ takes values in $[0,1]$, we have $G_{t,x,r}(0)\leq0$ and $G_{t,x,r}(1)\geq0$. As $u\mapsto G_{t,x,r}(u)$ is continuous, it follows that there is least one zero in $[0,1]$, and since the set of all zeros is compact, it has a maximum.
   
   We write $\rho(t,x,r)$ for the maximal zero of $G_{t,x,r}$. As $Y$ is increasing, the function $t\mapsto F_{t}(y)$ is decreasing and then so is $t\mapsto G_{t,x,r}(u)$; note that $g^{-1}$ is decreasing in $(t,x,u)$ and increasing in $r$. Hence, if $s\leq t$, the fact that $G_{t,x,r}>0$ on $(\rho(t,x,r),1]$ implies that $G_{s,x,r}>0$ on $(\rho(t,x,r),1]$ and hence that $\rho(s,x,r)\leq \rho(t,x,r)$. The monotonicity in $x$ and $r$ follows analogously.
   
   Let $t_{n}\downarrow t$ and  $x_{n}\downarrow x$ and $r_{n}\uparrow r$. Set $\rho_{n}=\rho(t_{n},x_{n},r_{n})$ and $\rho_{*}=\rho(t,x,r)$. By the above, $\rho_{n}$ is decreasing and $\rho_{n}\geq \rho_{*}$. Thus, we only need to verify that $\rho_{\infty}:=\lim \rho_{n}\leq \rho_{*}$. In view of the definition of $\rho_{*}$ as a maximal zero, it suffices to show that $\rho_{\infty}$ is a zero of $G_{t,x,r}$, and as $G_{t_{n},x_{n},r_{n}}(\rho_{n})=0$, that will follow if 
\begin{equation}\label{eq:eq:claimRC}
  (t,x,-r,u)\mapsto G_{t,x,r}(u)\quad \mbox{is jointly right-continuous.}
\end{equation}

Indeed, $y\mapsto F_{t}(y)$ is continuous, and together with the right-continuity of $Y$, it follows that $t\mapsto F_{t}(y)$ is right-continuous. %
Using also the continuity of $g^{-1}$, we see that~\eqref{eq:eq:claimRC} holds as desired. This completes the proof of the claim on $\rho$.
   
   (b) Next, we verify the equilibrium conditions. As a result of (a), the processes $t\mapsto\rho_{t}=\rho(t,X_{t},r(t,X_{t}))$ and $\gamma^{i}$ are increasing and right-continuous, and Lemma~\ref{le:singleAgent} yields that $\tau^{i}\in\cT^{i}$ is an optimal stopping time. Note that $(i,\omega)\mapsto\gamma^{i}(\omega)$ is $\Sigma$-measurable, and so is $(i,\omega)\mapsto r_{t}(\omega)$. Thus, $\{\tau^{i}\leq t\}=\{\gamma^{i}\geq r_{t}\}\in\Sigma$ for all $t\geq0$. Using the Conditional Exact Law of Large Numbers of Proposition~\ref{pr:condLLN}, the continuity of $y\mapsto F_{t}(y)$ and the definition of~$\rho_{t}$, we have $P$-a.s.\ that
  \begin{align*}
    \bar\rho_{t}:=\lambda\{i:\, \tau^{i}\leq t\} 
    &= \lambda\{i:\, g(t,X_{t},Y^{i}_{t},\rho(t,X_{t},r(t,X_{t})))\geq r(t,X_{t})\}\\
    &= \int P\{ g(t,X_{t},Y^{i}_{t},\rho(t,X_{t},r(t,X_{t})))\geq r(t,X_{t}) | X_{t}\}\,\lambda(di)\\
    &= 1-F_{t}(g^{-1}(t,X_{t},r(t,X_{t}),\rho(t,X_{t},r(t,X_{t}))))\\
    &=\rho(t,X_{t},r(t,X_{t}))=\rho_{t}.
  \end{align*}

   (c) Let $\bar\rho$ be a right-continuous process corresponding to an equilibrium; that is, $\bar\rho_{t}=\lambda\{i:\, \tau^{i}\leq t\}$ for some optimal $\tau^{i}\in\cT^{i}$. Then $\bar\rho$ is clearly increasing and $[0,1]$-valued. Due to the strict increase of $\gamma^{i}$, we know from Lemma~\ref{le:singleAgent} that $\tau^{i}=\inf\{t: \gamma^{i}_{t}\geq r_{t}\}$, which also ensures that $\{\tau^{i}\leq t\}\in\Sigma$. Since we have assumed that 
  $\bar\rho_{t}=\bar\rho(t,X_{t},r(t,X_{t}))$, we obtain as in~(b) that
  $$
    \bar\rho_{t}=\lambda\{i:\, \tau^{i}\leq t\} 
    = 1-F_{t}(g^{-1}(t,X_{t},r(t,X_{t}),\bar\rho(t,X_{t},r(t,X_{t}))))\quad P\as
  $$
  for all $t\geq0$.
\end{proof}

\begin{remark}\label{rk:MarkovUniqueness}
  The result in Theorem~\ref{th:general}\,(ii) assumes a priori that the equilibrium $\bar\rho_{t}$ is Markovian; that is, a deterministic function of $(t,X_{t})$.
  
  (i) First, let us observe that this is not automatically the case: randomized equilibria may exist.  Consider the setting of Example~\ref{ex:threeSolutions} where $X$ is deterministic and Equation~\eqref{eq:general} has several (deterministic, increasing, right-continuous) solutions; in particular, a maximal solution $\rho(t)$ and a minimal solution $\rho'(t)$. Suppose that there is a Poisson process $N$ which is $\G^{i}$-adapted and independent of $Y^{i}$ for all $i$, and let $\sigma$ be its first jump time. Then,
  $$
    \bar\rho_{t} = \1_{[0,\sigma)}(t)\rho'(t) + \1_{[\sigma,\infty)}(t)\rho(t)
  $$
  defines another right-continuous, increasing solution of~\eqref{eq:general} which determines an equilibrium. However, $\bar\rho_{t}$ is not of the mentioned Markovian form. Instead, the agents can agree to change their behavior according to the independent randomization $\sigma$. We also refer to \cite{CarmonaDelaRueLacker.13, Lacker.14} for further insights on randomized (or ``weak'') equilibria in the context of standard mean field games.
  
  (ii) Second, let us show that the phenomenon mentioned in (i) cannot occur if uniqueness holds in Equation~\eqref{eq:general}. In the setting of Theorem~\ref{th:general}\,(ii), even if we do not suppose a priori that $\bar\rho$ is a function of $(t,X_{t})$, we have
  \begin{align*}
    \bar\rho_{t}=\lambda\{i:\, \tau^{i}\leq t\} 
    &= \lambda\{i:\, g(t,X_{t},Y^{i}_{t},\bar\rho_{t})\geq r(t,X_{t})\}\\
    &= \int P\{ g(t,X_{t},Y^{i}_{t},\bar\rho_{t})\geq r(t,X_{t}) | X_{t}\}\,\lambda(di)\\
    &= 1-F_{t}(g^{-1}(t,X_{t},r(t,X_{t}),\bar\rho_{t})).
  \end{align*}
  If $\rho(t,x,r)$ is the maximal solution of \eqref{eq:general} as constructed in the theorem and uniqueness holds for~\eqref{eq:general}, it follows that
  $$
    \bar\rho_{t} = \rho(t,X_{t},r(t,X_{t}))\quad P\as
  $$
  and in that sense, $\bar\rho$ is necessarily of the Markovian form.
\end{remark}

\begin{remark}\label{rk:conditioningOnPath}
  In the literature on mean field games driven by stochastic differential equations, the private states at time $t$ are usually independent conditionally on the whole path $(X_{s})_{s\leq t}$ of the common noise before time~$t$. In the present setting, we have assumed that the intensities $\gamma^{i}$ depend on~$X$ in a Markovian way, and hence it is sufficient to condition on the current value $X_{t}$. One could envision a similar result where $\gamma^{i}$ depends on $X$ in a path-dependent way, and then one would condition on the whole past of~$X$.
\end{remark}

\subsection{Additive Model and Noisy Observations}\label{se:additiveModel}

In this section, we enhance the toy model from Proposition~\ref{pr:privateInfo} by incorporating a public signal and obtain a tractable specification of the general model from Theorem~\ref{th:general}. 
Consider the setup introduced in Section~\ref{se:basicSetup} with atomless probability spaces $(I,\cI,\lambda)$ and $(\Omega,\cF,P)$, and let $(I\times\Omega,\Sigma,\mu)$ be a Fubini extension of their product. For each $i\in I$, let $Y^{i}\geq0$ be a right-continuous, increasing, measurable process. We assume that for each $t\geq0$,  $(i,\omega)\mapsto Y^{i}_{t}(\omega)$ is $\Sigma$-measurable and that $Y^{i}_{t}$, $i\in I$ are essentially pairwise independent. Moreover, we assume that the distribution of $Y^{i}_{t}(\cdot)$ has no atoms; that is, its c.d.f.\ $F_{t}$ is continuous. Furthermore, let $X\geq0$ be a right-continuous, increasing, measurable process such that $X_{t}$ and $Y^{i}_{t}$ are independent for all $t\geq0$. 
We take $r\in\R$ to be constant (for simplicity)  and 
$$
  \gamma^{i}_{t} = X_{t} + Y^{i}_{t} + c\rho_{t},
$$
where $c\geq0$ is a constant governing the strength of interaction; see also Example~\ref{ex:uniformInitial}. For the information structure, we may consider two cases. Either we see $\G^{i}$ as given and assume that 
\begin{itemize}
\item $X$ and $Y^{i}$ are $\G^{i}$-progressively measurable for all $i\in I$,
\end{itemize}
which was the point of view taken above. Or, we model that the agents observe only $X+Y^{i}$ and $\rho$, and thus we convene that
\begin{itemize}
  \item $\G^{i}$ is the right-continuous filtration generated by $X+Y^{i}$ and $\rho$, for all $i\in I$.
\end{itemize}
This allows for the interpretation of $X$ as a ``true signal,'' whereas agent $i$ can only observe the noisy signal $X+Y^{i}$ with i.i.d.\ noise $Y^{i}$. Although the agents have more information in the first setting, both yield the same equilibria---the form of $\tau^{i}$ stated below shows that the agents only use the observation of $\gamma^{i}$.
Indeed, Theorem~\ref{th:general} yields the following.

\begin{corollary}\label{co:additiveModel}
  The equation
  $$
    1-u = F_{t}(r-x-cu),\quad u\in [0,1]
  $$
  has a maximal solution $\rho(t,x)\in[0,1]$ for every $(t,x)\in\R_{+}\times\R$, and $\rho_{t}:=\rho(t,X_{t})$ is a right-continuous process. Define also 
  $$
    \gamma^{i}_{t} = X_{t} + Y^{i}_{t} + c\rho_{t}, \quad \tau^{i} = \inf \{t:\, X_{t} + Y^{i}_{t} + c\rho_{t} =r\}
  $$  
  and assume that~\eqref{eq:intCond} is satisfied for all $i$.
  
  (i) Then, $\rho$ and $(\tau^{i})_{i\in I}$ define an equilibrium: $\tau^{i}\in\cT^{i}$ is an optimal stopping time for agent $i$, the mapping $(i,\omega)\mapsto \tau^{i}(\omega)$ is $\Sigma$-measurable, and 
  $$
    \lambda\{i:\,\tau^{i}\leq t\}=\rho_{t}\quad P\as \quad \mbox{for all} \quad t\geq0.
  $$  
  
  (ii) Conversely, let $t\mapsto \bar\rho_{t}$ be a right-continuous process corresponding to an equilibrium and suppose that $\bar\rho_{t}=\bar\rho(t,X_{t})$ for some measurable function $\bar\rho$. If $\gamma^{i}$ is strictly increasing for all $i$, then for every $t\geq0$, $\bar\rho(t,x)$ solves~\eqref{eq:general} for $(P\circ X_{t}^{-1})$-almost all $x\in\R$.
\end{corollary}

A solvable example can be constructed along the lines of Example~\ref{ex:uniformInitial}.

\begin{example}\label{ex:uniformInitialCommonNoise}
	Let $r\geq1$ and let $U^{i}$, $i\in I$ be essentially pairwise i.i.d.\ with a uniform distribution on $[r-1,r]$ and such that $U^{i}$ and $X_{t}$ are independent for all $t\geq0$. Moreover, suppose that $X$ is strictly increasing with $X_{0}=0$ and $X_{\infty}>1$. For $c\in(0,1)$, consider the intensity process
	$$
	  \gamma^{i}_{t} = X_{t}+ U^{i} + c\rho_{t}.
	$$
	Then, the equation has a unique solution $\rho(t,x)$, and
	$$
	  \rho(t,X_{t})=[(1-c)^{-1}X_{t}]\wedge 1
	$$
	corresponds to the unique (Markovian) equilibrium. In particular, this equilibrium evolves in a nondegenerate way as long as $X$ does.
\end{example}

\newcommand{\dummy}[1]{}


\begin{thebibliography}{10}

\bibitem{Aumann.64}
R.~J. Aumann.
\newblock Markets with a continuum of traders.
\newblock {\em Econometrica}, 32:39--50, 1964.

\bibitem{Bardi.12}
M.~Bardi.
\newblock Explicit solutions of some linear-quadratic mean field games.
\newblock {\em Netw. Heterog. Media}, 7(2):243--261, 2012.

\bibitem{BardiPriuli.14}
M.~Bardi and F.~S. Priuli.
\newblock Linear-quadratic {$N$}-person and mean-field games with ergodic cost.
\newblock {\em SIAM J. Control Optim.}, 52(5):3022--3052, 2014.

\bibitem{BensoussanFrehseYam.13}
A.~Bensoussan, J.~Frehse, and S.~C.~P. Yam.
\newblock {\em Mean field games and mean field type control theory}.
\newblock Springer Briefs in Mathematics. Springer, New York, 2013.

\bibitem{BensoussanSungYamYung.16}
A.~Bensoussan, K.~C.~J. Sung, S.~C.~P. Yam, and S.~P. Yung.
\newblock Linear-{Q}uadratic {M}ean {F}ield {G}ames.
\newblock {\em J. Optim. Theory Appl.}, 169(2):496--529, 2016.

\bibitem{Cardaliaguet.13}
P.~Cardaliaguet.
\newblock Notes on mean field games (from {P.-L.} {L}ions' lectures at
  {C}oll{\`e}ge de {F}rance).
\newblock 2013.

\bibitem{CarmonaDelaRue.13}
R.~Carmona and F.~Delarue.
\newblock Probabilistic analysis of mean-field games.
\newblock {\em SIAM J. Control Optim.}, 51(4):2705--2734, 2013.

\bibitem{CarmonaDelaRue.14}
R.~Carmona and F.~Delarue.
\newblock The master equation for large population equilibriums.
\newblock In {\em Stochastic analysis and applications 2014}, volume 100 of
  {\em Springer Proc. Math. Stat.}, pages 77--128. Springer, Cham, 2014.

\bibitem{CarmonaDelaRue.15}
R.~Carmona and F.~Delarue.
\newblock Forward-backward stochastic differential equations and controlled
  {M}c{K}ean-{V}lasov dynamics.
\newblock {\em Ann. Probab.}, 43(5):2647--2700, 2015.

\bibitem{CarmonaDelaRueLacker.13}
R.~Carmona, F.~Delarue, and D.~Lacker.
\newblock Mean field games with common noise.
\newblock {\em Ann. Probab.}, 44(6):3740--3803, 2015.

\bibitem{CarmonaDelaRueLacker.17}
R.~Carmona, F.~Delarue, and D.~Lacker.
\newblock Mean field games of timing and models for bank runs.
\newblock {\em Appl. Math. Optim.}, 76(1):217--260, 2017.

\bibitem{CarmonaFouqueSun.13}
R.~Carmona, J.-P. Fouque, and L.-H Sun.
\newblock Mean field games and systemic risk.
\newblock {\em Commun. Math. Sci.}, 13(4):911--933, 2015.

\bibitem{CarmonaLacker.15}
R.~Carmona and D.~Lacker.
\newblock A probabilistic weak formulation of mean field games and
  applications.
\newblock {\em Ann. Appl. Probab.}, 25(3):1189--1231, 2015.

\bibitem{DiamondDybvig.83}
D.~W. Diamond and P~H. Dybvig.
\newblock Bank runs, deposit insurance, and liquidity.
\newblock {\em J. Polit. Econ.}, 91(3):401--419, 1983.

\bibitem{DuffieSun.10}
D.~Duffie and Y.~Sun.
\newblock The exact law of large numbers for independent random matching.
\newblock {\em J. Econom. Theory}, 147(3):1105--1139, 2012.

\bibitem{Fischer.14}
M.~Fischer.
\newblock On the connection between symmetric {$N$}-player games and mean field
  games.
\newblock {\em Ann. Appl. Probab.}, 27(2):757--810, 2017.

\bibitem{GamesSaude.14}
D.~A. Gomes and J.~Sa{\'u}de.
\newblock Mean field games models---a brief survey.
\newblock {\em Dyn. Games Appl.}, 4(2):110--154, 2014.

\bibitem{GueantLasryLions.11}
O.~Gu{\'e}ant, J.-M. Lasry, and P.-L. Lions.
\newblock Mean field games and applications.
\newblock In {\em Paris-{P}rinceton {L}ectures on {M}athematical {F}inance
  2010}, volume 2003 of {\em Lecture Notes in Math.}, pages 205--266. Springer,
  Berlin, 2011.

\bibitem{HuangMalhameCaines.07}
M.~Huang, P.~E. Caines, and R.~P. Malham{\'e}.
\newblock Large-population cost-coupled {LQG} problems with nonuniform agents:
  individual-mass behavior and decentralized {$\epsilon$}-{N}ash equilibria.
\newblock {\em IEEE Trans. Automat. Control}, 52(9):1560--1571, 2007.

\bibitem{HuangMalhameCaines.06}
M.~Huang, R.~P. Malham{\'e}, and P.~E. Caines.
\newblock Large population stochastic dynamic games: closed-loop
  {M}c{K}ean-{V}lasov systems and the {N}ash certainty equivalence principle.
\newblock {\em Commun. Inf. Syst.}, 6(3):221--251, 2006.

\bibitem{KaratzasShubikSudderth.94}
I.~Karatzas, M.~Shubik, and W.~D. Sudderth.
\newblock Construction of stationary {M}arkov equilibria in a strategic market
  game.
\newblock {\em Math. Oper. Res.}, 19(4):975--1006, 1994.

\bibitem{Lacker.14}
D.~Lacker.
\newblock A general characterization of the mean field limit for stochastic
  differential games.
\newblock {\em Probab. Theory Related Fields}, 165(3-4):581--648, 2016.

\bibitem{Lando.09}
D.~Lando.
\newblock {\em Credit Risk Modeling: Theory and Applications}.
\newblock Princeton University Press, Princeton, 2009.

\bibitem{LasryLions.06a}
J.-M. Lasry and P.-L. Lions.
\newblock Jeux \`a champ moyen. {I}. {L}e cas stationnaire.
\newblock {\em C. R. Math. Acad. Sci. Paris}, 343(9):619--625, 2006.

\bibitem{LasryLions.06b}
J.-M. Lasry and P.-L. Lions.
\newblock Jeux \`a champ moyen. {II}. {H}orizon fini et contr\^ole optimal.
\newblock {\em C. R. Math. Acad. Sci. Paris}, 343(10):679--684, 2006.

\bibitem{LasryLions.07}
J.-M. Lasry and P.-L. Lions.
\newblock Mean field games.
\newblock {\em Jpn. J. Math.}, 2(1):229--260, 2007.

\bibitem{MorrisShin.04}
S.~Morris and H.~S. Shin.
\newblock Coordination risk and the price of debt.
\newblock {\em Eur. Econ. Rev}, 48(1):133--153, 2004.

\bibitem{Pham.16}
H.~Pham.
\newblock Linear quadratic optimal control of conditional {McKean-Vlasov}
  equation with random coefficients and applications.
\newblock {\em Preprint arXiv:1604.06609v1}, 2016.

\bibitem{Podczeck.10}
K.~Podczeck.
\newblock On existence of rich {F}ubini extensions.
\newblock {\em Econom. Theory}, 45(1-2):1--22, 2010.

\bibitem{QiaoSunZhang.14}
L.~Qiao, Y.~Sun, and Z.~Zhang.
\newblock Conditional exact law of large numbers and asymmetric information
  economies with aggregate uncertainty.
\newblock {\em Econom. Theory}, pages 1--22, 2014.

\bibitem{Sun.06}
Y.~Sun.
\newblock The exact law of large numbers via {F}ubini extension and
  characterization of insurable risks.
\newblock {\em J. Econom. Theory}, 126(1):31--69, 2006.

\bibitem{SunZhang.09}
Y.~Sun and Y.~Zhang.
\newblock Individual risk and {L}ebesgue extension without aggregate
  uncertainty.
\newblock {\em J. Econom. Theory}, 144(1):432--443, 2009.

\end{thebibliography}
\end{document}